\documentclass[10pt,a4paper]{article}
\usepackage{amsmath,amssymb,amsfonts,amsthm}
\usepackage{float,graphicx,color}
\usepackage[all]{xy}

\newcommand{\rmd}{\mathrm{d}}

\newcommand{\rmD}{\mathrm{D}}

\newcommand{\bbD}{\mathbb{D}}
\newcommand{\bbE}{\mathbb{E}}
\newcommand{\bbF}{\mathbb{F}}
\newcommand{\bbG}{\mathbb{G}}

\newcommand{\bbN}{\mathbb{N}}

\newcommand{\bbR}{\mathbb{R}}

\newcommand{\bbZ}{\mathbb{Z}}

\newcommand{\frg}{\mathfrak{g}}
\newcommand{\frs}{\mathfrak{s}}

\newcommand{\fro}{\mathfrak{o}}

\newcommand{\calF}{\mathcal{F}}

\newcommand{\calO}{\mathcal{O}}
\newcommand{\calP}{\mathcal{P}}

\newcommand{\calT}{\mathcal{T}}

\DeclareMathOperator{\End}{End}

\newcommand{\beq}{\begin{equation}}
\newcommand{\eeq}{\end{equation}}

\newcommand{\id}{\mathrm{id}}

\newcommand{\gt}{GT}
\newcommand{\pt}{PT}
\newcommand{\hol}{\mathrm{Hol}}

\newtheorem{theorem}{Theorem}[section]
\newtheorem{lemma}[theorem]{Lemma}
\newtheorem{corollary}[theorem]{Corollary}
\newtheorem{proposition}[theorem]{Proposition}

\theoremstyle{definition}

\theoremstyle{remark}
\newtheorem{remark}{Remark}[section]

\theoremstyle{plain}

\begin{document}

\title{Exact formulas for the approximation of connections and curvature}

\author{Snorre H. Christiansen\footnote{Department of Mathematics, University of Oslo, PO box 1053 Blindern, NO-0316 Oslo, Norway. email: {\tt snorrec@math.uio.no} }}

\date{}

\maketitle

\begin{abstract}
First we express the holonomy along a boundary curve as the integral on the domain, of an expression which is linear in the curvature. Then we provide a rigorous justification of the definition of curvature in Regge calculus.
\end{abstract}

\bigskip

\begin{flushright}
\emph{Je t'apporte l'enfant d'une nuit d'Idum\'ee!}\\
Don du po\`eme, Mallarm\'e.
\end{flushright}

\section{Introduction}

We present two results on connections and curvature that aim to relate the continuous and the discrete. Whether nature is one or the other, remains open.

The first result was inspired by the desire to extend the Lattice Gauge Theory initiated by Wilson \cite{Wil74}, to a higher order method. While we did not quite achieve this goal, a formula was obtained, that might be of independent interest. It expresses the holonomy around a closed curve as an exact integral which is linear in the curvature. This continues our earlier investigations on LGT \cite{ChrHal09BIT}\cite{ChrHal11IMA}\cite{ChrHal11SINUM}\cite{ChrHal12JMP}, which were concerned with convergence analysis, mainly when the gauge field describes electromagnetism, and extending the method to simplicial meshes (rather than the cubical ones that are customary).

The second result is a justification of the definition of curvature in the calculus of Regge \cite{Reg61}. Those provided in \cite{Reg61} and \cite{FriLee84} were not found to be completely rigorous. Earlier \cite{Chr04M3AS}\cite{Chr11NM}, we have related Regge calculus to finite elements  and studied linearization. Here we present a result on the non-linear method.

\section{Definitions}
We present here some some notions on connections and curvature, to fix notations. A standard reference on the subject is \cite{KobNom63}. We have mainly used \cite{Tay96II} (Appendix C).

Let $V$ be a finite dimensional Euclidean vector space. The space of endomorphisms of $V$ (that is, linear maps $V \to V$) is denoted $\End(V)$. Let $\bbG$ be a closed subgroup of the orthogonal endomorphims of $V$ and $\frg$ its associated Lie algebra. 

Given a function $Q: S \to \bbG$ one transforms elements of $\Omega^k(S) \otimes V$ as follows:
\begin{equation}
\Phi \mapsto Q \Phi.
\end{equation}
One also transforms connection one-forms $A \in \Omega^1(S) \otimes \frg$ as follows:
\begin{equation}
A \mapsto \gt_Q(A) = Q A Q^{-1} - (\rmD Q) Q^{-1}.
\end{equation}
This formula ensures that we have:
\begin{equation}
\nabla_{\gt_Q(A)}{Q\Phi} = Q(\nabla_A \Phi).
\end{equation}

Parallel transport with respect to $A$, along a curve $\gamma: [a,b] \to S$, from $x$ to $y$, is denoted:
\begin{equation}
\pt_A(\gamma).
\end{equation}
It is defined as the linear map $V \to V$, which to a vector $u(x)\in V $, associates $u(y) \in V$ in such a way that there is a field $u$, defined on $\gamma$, that satisfies:
\begin{equation}
\nabla_{A} u (\dot \gamma) = 0.
\end{equation}
In the commutative case we have the formula:
\begin{equation}
\pt_A(\gamma)= \exp( -\int_\gamma A ).
\end{equation}
If the endpoint $y$ of $\gamma$ is  also its origin $x$, one speaks of a holonomy, and we denote it by:
\begin{equation}
\hol_A(\gamma).
\end{equation}

Parallel transport along a curve $\gamma$ from $x$ to $y$, behaves as follows under gauge transformations:
\begin{equation}\label{eq:gtpt}
\pt_{\gt_Q(A)} (\gamma) Q(x) = Q(y)\pt_{A} (\gamma). 
\end{equation}
In particular, around a closed curve from $x$ to $x$, we get:
\begin{equation}\label{eq:holtransf}
\hol_{\gt_Q(A)} (\gamma) = Q(x) \hol_{A}(\gamma)  Q(x)^{-1}.
\end{equation}

The curvature of $A$ is denoted $\calF(A)$:
\begin{equation}
\calF(A) = \rmd A + 1/2 [A, A].
\end{equation}
We have:
\begin{equation}
\calF(\gt_Q(A))= Q \calF(A) Q^{-1}.
\end{equation}

\section{Holonomy from curvature}

It is well known that the holonomy around a curve, minus the identity, is a good approximation of the integral of the curvature on the surface the curve bounds, in the sense that the difference between the two is smaller by one order of the length of the curve, see e.g. \cite{Tay96II} (Appendix C, Proposition 5.1). That is, for small domains $T$:
\begin{equation}\label{eq:holfromfest}
\hol_A(\partial T) - I  =  - \int_T \calF(A)(x) \rmd x + \calO( (\textrm{area}( T))^{3/2} ).
\end{equation}
This fact is the basis for Lattice Gauge Theory, introduced in \cite{Wil74}. See for instance in \cite{ChrHal12JMP}, how Proposition 2 is used as an ingredient to prove consistency. It turns out that in discretizations, the left hand side has better invariance properties than the right hand side, under discrete gauge transformations. Discrete gauge invariance is a crucial property, linked to charge conservation by Noether's theorem.

In the next proposition we transform this estimate into an exact identity, expressing the holonomy as an integral, which is linear in the curvature. The estimate (\ref{eq:holfromfest}) can easily be deduced from the proposed identity. The original motivation was construct a discretely gauge invariant discretization of Yang-Mills action, with higher orders of convergence than classical LGT. In this we have not yet succeeded.

\begin{proposition}
Suppose $T$ is an oriented rectangle. Define, for any $x \in T$,  two paths, $\gamma^-(x)$ and $\gamma^+(x)$, as follows. The face is equipped with two coordinates determined by the axes of $T$, compatible with its orientation. The origin of $T$ has coordinates $(0,0)$ and the opposite vertex in $T$ has coordinates $(a,b)$. We put $x = (x_0, x_1)$ and let the paths consists of straight lines joining the following points: 
\begin{align}
\gamma^-(x)& : (0,0) \to (x_0, 0) \to (x_0, x_1),\\
\gamma^+(x)& : (x_0,x_1) \to (x_0, b) \to (0, b) \to (0,0).
\end{align}
Then we have:
\begin{equation}\label{eq:holfromf}
\hol_A(\partial T) - I  =  - \int_T \pt_A(\gamma^+(x)) \calF(A)(x) \pt_A (\gamma^-(x)) \rmd x.
\end{equation}
\end{proposition}
\begin{proof}
 (i) Remark first that if identify (\ref{eq:holfromf}) holds for a gauge potential $A$ then it holds for any gauge transformation $\gt_Q(A)$ of $A$.
 
 (ii) Given $x\in T$ define a path $\alpha(x)$ consisting of straight lines as follows:
 \begin{equation}
 \alpha: (x_0,x_1) \to (0,x_1) \to (0,0).
 \end{equation}
 The path, followed in reverse is denoted $\alpha(x)^{-1}$.
 Define $Q:f \to \bbG$ as follows: 
 \begin{equation}
 Q(x) = \pt_A (\alpha(x)).
 \end{equation}
 
The gauge potential $A' = \gt_Q(A)$ now satisfies, by (\ref{eq:gtpt}):
\begin{equation}
 \pt_{A'}(\alpha(x)^{-1}) = \id_V,
\end{equation}
therefore,  for all $x_0 \in [0,a]$, $x_1 \in [0,b]$:
 \begin{align}
 A'_0(x_0,x_1) & = 0,\\
 A'_1(0,x_1) & = 0.
 \end{align}
 
 (iii) For $A$ of the above form, the proposition is proved for fixed $b$, differentiating with respect to $a$. More precisely, for any point $(x_0,x_1) \in T$, let $P(x_0, x_1)$ be the parallel transport, according to $A$ along the segment from $(x_0,0)$ to $(x_0,x_1)$. 
We have:
\begin{equation}
\hol_A(\partial T) = P(a,b).
\end{equation}
We have:
\begin{equation}
\partial_1 P (x) = -A_1(x) P(x).
\end{equation}
We deduce:
\begin{equation}
\partial_1 (P(x)^{-1} \partial_0 P(x)) = -P(x)^{-1} \partial_0 A_1 (x) P(x).
\end{equation}
Hence:
\begin{equation}
\partial_0 P (x_0,b) = -\int_{0}^{b} P(x_0,b) P(x_0,x_1)^{-1} \partial_0 A_1 (x_0,x_1) P(x_0,x_1) \rmd x_1.  
\end{equation}
So that:
\begin{equation}
P(a,b) = I - \int_0^a \int_0^b  P(x_0,b) P(x_0,x_1)^{-1} \partial_0 A_1 (x_0,x_1) P(x_0,x_1) \rmd x_0 \rmd x_1.
\end{equation}
This can be interpreted as the claimed identity.
\end{proof}

\begin{remark}
In the abelian case one can give a much simpler proof of this identity. We consider the rectangle $T=T(a,b)$ as a function of the upper right corner.
Define:
\begin{align}
 H(a,b) &= \hol_A(\partial T (a,b)),\\
 &= \exp(-\int_{\partial T(a,b)} A),\\
 &= \exp(-\int_{T(a,b)} \rmd A).
 \end{align}
From the last expression we deduce:
\begin{align}
 \partial_x H(x,b) = - H(x,b) \int_0^b \rmd A (x,y)\rmd y.
\end{align}
So we can write:
\begin{equation}
H(a,b) = I - \int_0^a \int_0^b H(x,b) \rmd A (x,y)\rmd x \rmd y,
\end{equation}
From the second to last expression on the other hand, we deduce that for $y \in [0,b]$ : 
\begin{equation}
 \pt_A(\gamma^+(x,y)) \pt_A (\gamma^-(x,y)) = H(x,b).\\
\end{equation}
so that we have obtained the desired identity.
\end{remark}

\section{A justification of Regge Calculus}

\subsection{Definition of Regge Caluclus}
Regge calculus \cite{Reg61} can be defined as follows.

Let $\calT$ be a simplicial complex, that is, a finite set of of finite non-empty sets. The elements of $\calT$ are called simplices and are thought of as sets of vertices. For each simplex $T\in \calT$, its geometric realization is the set:
\begin{equation}
|T| = \{ f : T \to \bbR \ : \ \sum_{x \in T} f(x) = 1, \ \forall x \in T \ f(x) \geq 0 \}.
\end{equation}
A vertex $x \in T$ can be identified with the characteristic function of $\{ x \}$ on $T$, which is an element of $|T|$. For $T' \subseteq T$ there is a unique affine map $\Phi_{TT'} : |T'| \to | T |$ which is the identity on vertices of $T'$.

The geometric realization of $\calT$ is :
\begin{equation}
\coprod_{T \in \calT} | T |  \big / \sim.
\end{equation}
where the equivalence relation is the smallest satisfying:
\begin{equation}
\Phi_{TT'} (x) \sim x, \textrm{ whenever }  x \in |T'| \textrm{ and }  T' \subseteq T.
\end{equation}
In particular the maps $\Phi_{TT'}$ are identified with inclusions.

Suppose that $|\calT| $ is an oriented $n$-dimensional manifold. In Regge calculus one assigns a real number to each edge. These numbers, interpreted at edge lengths squared, determine a constant metric $\rho$ on each simplex.

Then, to each codimension $2$ simplex $h$ (called hinge) in $\calT $ one associates a so-called deficit angle $d_h$ as follows. Compute, for each $n$-simplex in $\calT$ containing the hinge, the dihedral angle between the two faces arriving at the hinge. Add these dihedral angles, and substract this number from $2 \pi$, to get the deficit angle $d_h$. Let $a_h$ be the area of the hinge. The action defined by Regge to mimick the Einstein-Hilbert action is:
\begin{equation}
\rho \mapsto \sum_h d_h a_h.
\end{equation}
Critical point of this action are discrete analogues of Einstein metrics.

One goes even further and asserts that the scalar curvature is a measure on $|\calT |$ concentrated to the hinges and given by:
\begin{equation}
\psi \mapsto  \sum_h d_h \int_h \psi,
\end{equation}
where one sums over hinges $h$, the integrals of $\psi$ on $h$ equipped with the induced metric. The difficulty, in order to make sense of this assertion, is that scalar curvature is a non-linear expression of the metric involving second order derivatives. That it could be well defined for some discontinuous metrics is miraculous. 

Regge \cite{Reg61} proposed a justification involving an averaging argument and an appeal to the Gauss-Bonnet theorem. An alternative justification can be found in \cite{FriLee84}, based on an imbedding in higher dimensional vector space (see in particular \S 3 and Theorem 3.1).  A sequence of smooth metrics approximating the Regge metric is considered, and one wants to obtain the curvature of the limit (as defined by Regge) as the limit of curvatures (as usually defined for smooth metrics). However we don't think the passage to the limit is valid for all approximating sequences, and \cite{FriLee84} is vague about which approximating sequences are used. 

Ideally one might want to identify a topology on the space of metrics, with respect to which this amounts to continuity of the curvature map (into the space of measures). We have not identified such a topology. But in this paper we prove that the limiting procedure is valid for the canonical approximating sequences, obtained by smoothing by convolution.

Various results connecting integrals of curvature with holonomies can be found in Appendix C.5 of \cite{Tay96II}. Based on such considerations we are able to evaluate the curvature of the smoothed Regge metrics. Our arguments can be see to reprove a variant of the Gauss-Bonnet theorem. We hope the reader will share our pleasure in doing so.

\subsection{Justification in two dimensions}
Here we concentrate on the two-dimensional case ($n = 2$).

Let $\bbE$ be a two dimensional Euclidean vector space, whose metric is denoted $g$ and serves as a reference. Half-lines emanating from the origin split the space into $I \in \bbN$ sectors. The half-lines are indexed by a cyclic variable $i\in \bbZ / I \bbZ$. The sector between $i$ and $i+1$ is indexed by $i + 1/2$.

In this context we consider a Regge metric $\rho$. It is constant in each sector, with value in the sector $i +1/2$ denoted $\rho_{i+1/2}$. Its pullback to the half-lines separating two sectors is well-defined,  that is, the restriction is the same from both sides, when evaluated on vectors parallel to the half-line.

Let $m_i$ denote the directing vector of half-line $i$ which has unit length, with respect to $\rho$. Let $\theta_{i+1/2}$ be the angle between the vectors $m_i$ and $m_{i+1}$ with respect to the metric $\rho_{i+1/2}$. The deficit angle, at the origin, is defined to be:
\begin{equation}
d= 2\pi - \sum_i \theta_{i+1/2}.
\end{equation}

Choose $\phi$, a smooth function on $\bbE$ with compact support in the unit ball and with integral $1$, with respect to $g$. For $\epsilon >0$ define the scaling:
\begin{equation}
\phi_\epsilon (x)= \epsilon^{-2} \phi(\epsilon^{-1} x),
\end{equation}
where the factor in front is chosen to preserve the value of the integral. Define the smoothed Regge metrics $\sigma_\epsilon$ by the following convolution product, computed with respect to $g$:
\begin{equation}\label{eq:smoothmet}
\sigma_\epsilon = \phi_\epsilon \ast \rho.
\end{equation}

We concentrate first on the metric $\sigma = \sigma_1$. We denote by $\nabla$ the Levi-Civita connection, $\kappa$ the scalar curvature and by $\mu$ the volume two-form of $\sigma$. Our goal is to prove that the function $\kappa$ has compact support and:
\begin{equation}
\int_\bbE \kappa \mu = d.
\end{equation}
We do this by evaluating the holonomy (with respect to $\nabla$) along a curve encircling the origin, at sufficient distance, in two different ways.

Let $x \mapsto (e_1(x), e_2(x))$ denote a choice of orthonormal oriented basis (at $x \in \bbE$). Given this frame, denote by $A$ the connection one-form of the Levi-Civita connection of $\sigma$. Thus:
\begin{equation}
A \in \Omega^1(\bbE) \otimes \frs \fro (2) .
\end{equation}
The Lie algebra $\frs \fro (2)$ is one-dimensional and spanned by the matrix:
\begin{equation}
J = \begin{bmatrix}
0 & -1 \\
1 & 0
\end{bmatrix}.
\end{equation}

\begin{proposition}\label{prop:holone}
 Let $T$ be a domain in $\bbE$ with a piecewise smooth boundary curve $\partial T$.
We have:
\begin{equation}
\hol_A (\partial T) = \exp((\int_T \kappa \mu) J ).
\end{equation}
\end{proposition}
\begin{proof}
We denote by $F$ the curvature tensor, also in the frame $(e_1, e_2)$. We have:
\begin{equation}
F = \rmd A + \frac{1}{2} [A, A] = \rmd A \in \Omega^2(\bbE) \otimes \frs \fro (2).
\end{equation}
By Stokes theorem:
\begin{align}
\hol_A (\partial T)& = \exp (-\int_{\partial T} A ) =  \exp (-\int_{T} F ).
\end{align}
But we also have (equation (5.13) in \cite{Tay96II}):
\begin{equation}
F= - \kappa J \mu.
\end{equation}
This concludes the proof.
\end{proof}

Next we compute the holonomy by parallel transporting along the curve. This is made easy by the following fact: 
\begin{lemma}\label{lem:nablam}
In the union of the sectors $i - 1/2$ and $i+ 1/2$, consider the subset $U_i$ of points whose distance to the boundary is strictly larger than $1$.

On $U_i$ we have $\nabla m_i = 0$.
\end{lemma}
\begin{proof}
For simplicity of notation, we write $m=m_i$. Recall that for any constant vector fields $X,Y$ on $U_i$, we have:
\begin{equation}
2 \sigma ( \nabla_X m , Y) = \partial_X \sigma(m, Y) + \partial_m \sigma(X,Y) - \partial_Y \sigma(m, X).
\end{equation}
Here, a vectorfields $Z$ acts on scalar fields as derivations, denoted $\partial_Z$.

Denote by $n$ a vector which is orthogonal to $m$ for the reference metric $g$. We use that $\sigma$ is invariant in the $m$ direction and that $\sigma(m,m)$ is constant, and compute:
\begin{align}
2 \sigma ( \nabla_m m, m) & = \partial_m \sigma(m, m) + \partial_m \sigma(m,m) - \partial_m \sigma(m, m) = 0,\\
2 \sigma ( \nabla_m m, n) & = \partial_m \sigma(m, n) + \partial_m \sigma(m,n) - \partial_n \sigma(m, m) = 0,\\
2 \sigma ( \nabla_n m, m) & = \partial_n \sigma(m, m) + \partial_m \sigma(n,m) - \partial_m \sigma(m, n) = 0,\\
2 \sigma ( \nabla_n m, n) & = \partial_n \sigma(m, n) + \partial_m \sigma(n,n) - \partial_n \sigma(m, n) = 0.
\end{align}
This concludes the proof.
\end{proof}

Let $\alpha_{i+1/2}$ be the angle at the origin of the sector $i+1/2$, computed with respect to the reference metric $g$. Elementary trigonometry shows that the union of the domains $U_i$ contains the exterior of the ball with radius:
\begin{equation}
r=\max_i 1/\cos (\pi /2 - \alpha_{i + 1/2} /2).
\end{equation}

From the preceding Lemma one gets:
\begin{corollary}  The scalar curvature $\kappa$ is supported in the ball  $B_g(0,r)$.
\end{corollary}

\begin{proposition} \label{prop:holtwo} Let $T$ be a domain containing the ball $B_g(0,r)$. We have:
\begin{equation}
\hol_A(\partial T) = \exp (- (\sum_i \theta_{i+1/2}) J ).
\end{equation}
\end{proposition}

\begin{proof}
For each $i$, define vectors $n_i^+$ and $n_i^-$ such that $(m_i, n_i^\pm)$ is an orthonormal oriented basis with respect to the metric $\rho_{i\pm 1/2}$.

Also, in each sector $i+1/2$, choose a point $p_{i+1/2}$ on the boundary curve, such that:
\begin{equation}
p_{i+1/2} \in U_i \cap U_{i+1}.
\end{equation}
Let $\gamma_i$ be the portion of the boundary curve from $p_{i-1/2}$ to $p_{i+1/2}$, inside $U_i$.

From Lemma \ref{lem:nablam} it follows that:
\begin{equation}
\pt_{\nabla}(\gamma_i): m_i \mapsto m_i,
\end{equation}
and then, since parallel transport along $\gamma_i$, is an isometry from the metric $\sigma$ at $p_{i-1/2}$ (which is equal to $\rho_{i-1/2}$) to the metric $\sigma$ at $p_{i+1/2}$ (which is equal to $\rho_{i+1/2}$), it follows that:
\begin{equation}
\pt_{\nabla}(\gamma_i): n_i^- \mapsto n_i^+.
\end{equation}
The matrix of the identity from the basis $(m_i, n_i^+)$ to the basis $(m_{i+1}, n_{i+1}^-)$ (both of which are orthonormal oriented for $\rho_{i+1/2}$) is:
\begin{equation}
\begin{bmatrix}
\cos -\theta_{i+1/2} & - \sin -\theta_{i+1/2}\\
\sin -\theta_{i+1/2}  & \cos -\theta_{i+1/2} 
\end{bmatrix}.
\end{equation}

We write:
\begin{equation}
\pt_{\nabla}(\partial T)  =  \pt_{\nabla}(\gamma_{I-1}) \circ \ldots \circ \pt_{\nabla}(\gamma_0).
\end{equation}
Expressed in the basis $(m_0, n_0^-)$, attached to the point  $p_{-1/2}$, the right hand side evaluates to:
\begin{equation}
\begin{bmatrix}
\cos -\theta_{I-1/2} & - \sin -\theta_{I- 1/2}\\
\sin -\theta_{I- 1/2}  & \cos -\theta_{I-1/2} 
\end{bmatrix}
\ldots 
\begin{bmatrix}
\cos -\theta_{1/2} & - \sin -\theta_{1/2}\\
\sin -\theta_{1/2}  & \cos -\theta_{1/2} 
\end{bmatrix} = \exp (- (\sum_i \theta_{i+ 1/2} ) J ).
\end{equation}

Since the expression for the holonomy is the same in the basis $(e_1(p_{-1/2}), e_2(p_{-1/2}) )$, the claimed identity follows.
\end{proof}

We are now ready to conclude:
\begin{proposition} \label{prop:intkappa} We have:
\begin{equation}
\int_\bbE \kappa \mu = d.
\end{equation}
\end{proposition}
\begin{proof}
From Propositions \ref{prop:holone} and \ref{prop:holtwo} we deduce:
\begin{equation}\label{eq:almost}
\int_\bbE \kappa \mu +  \sum_i \theta_{i+1/2} \in 2\pi \bbZ.
\end{equation}
Next we consider the following one-parameter family of Regge metrics:
\begin{equation}
[0,1] \ni s \mapsto \rho(s) = s\rho + (1-s)g.
\end{equation}
The left hand side in (\ref{eq:almost}), evaluated with $\rho$ replaced by $\rho(s)$, varies continuously as a function of $s$, and takes discrete values, so must be constant. Moreover at $s=0$ one obtains $2\pi$. Therefore the value at $s=1$ is also $2\pi$.
\end{proof}

We now return to the family of smoothed metrics $\sigma_\epsilon$ defined by (\ref{eq:smoothmet}). We let $\kappa_\epsilon$ and $\mu_\epsilon$ denote their respective scalar curvatures and volume forms.
\begin{proposition}
We have:
\begin{equation}
\kappa_\epsilon \mu_\epsilon \to d \delta,
\end{equation}
in the sense that for any continuous function $\psi$:
\begin{equation}
\lim_{\epsilon \to 0} \int \psi \kappa_\epsilon \mu_\epsilon = d \psi(0).
\end{equation}
\end{proposition}
\begin{proof}
Let $\Phi_\epsilon:\bbE \to \bbE$ be the scaling map:
\begin{equation}
\Phi_\epsilon(x) = \epsilon^{-1} x.
\end{equation}
Since:
\begin{equation}
\rho = \epsilon^2 \Phi_{\epsilon}^\star \rho,
\end{equation}
we get:
\begin{equation}
\sigma_\epsilon = \epsilon^2 \Phi_{\epsilon}^\star \sigma.
\end{equation}
It follows that:
\begin{align}
\kappa_\epsilon &= \epsilon^{-2}\Phi_{\epsilon}^\star \kappa,\\
\mu_\epsilon & = \epsilon^2 \Phi_{\epsilon}^\star \mu.
\end{align}
So that:
\begin{equation}
\kappa_\epsilon \mu_\epsilon = \Phi_{\epsilon}^\star (\kappa \mu).
\end{equation}

Based on this identity, the convergence follows.
\end{proof}

\subsection{Justification in higher dimensions}

We let $\bbE$ denote some Euclidean space of dimension at least three. Its metric is denoted $g$ and is used to define smoothing by convolution.

Let $\bbF$ be a subspace of codimension two. We devide $\bbE$ into a finite number $I$ of sectors around $\bbF$, by considering half-hyperplanes emanating from $\bbF$. These half-hyperplanes are indexed as before by $i \in \bbZ/I\bbZ$, and the sector between $i$ and $i+1$ is indexed by $i + 1/2$. 

We let $\rho$ be a metric on $\bbE$ defined as follows. In each sector it is constant and positive definite, with value denoted $\rho_{i + 1/2}$ in sector $i + 1/2$. Its pullback to the half-hyperplanes should be well defined, in the sense that the pullback by the canonical injection is the same from both sides. In other words $\rho$ is continuous across interfaces when applied to two tangential vectors. This is an analogue of a Regge metric in a simplified setting, where one just looks at what happens around a single hinge.

Let $m_i$ denote the vector in the half-hyperplane $i$ which has unit length and is orthogonal to $\bbF$, with respect to $\rho$. We choose the orientation that makes it point into the half-hyperplane.

 Let $\theta_{i+1/2}$ be the angle between the vectors $m_i$ and $m_{i+1}$ with respect to the metric $\rho_{i+1/2}$. The deficit angle, along the hinge $\bbF$, is defined to be:
\begin{equation}\label{eq:deficit}
d= 2\pi - \sum_i \theta_{i+1/2}.
\end{equation}

As before we choose $\phi$, a smooth function on $\bbE$ with compact support in the unit ball and with integral $1$, with respect to $g$. For $\epsilon >0$ define the scaling:
\begin{equation}
\phi_\epsilon (x)= \epsilon^{- \dim \bbE} \phi(\epsilon^{-1} x),
\end{equation}
where the factor in front is chosen to preserve the value of the integral. Define the smoothed Regge metrics $\sigma_\epsilon$ by the following convolution product, computed with respect to $g$:
\begin{equation}
\sigma_\epsilon = \phi_\epsilon \ast \rho.
\end{equation}
We let $\mu_\epsilon$ denote the volume form attached to $\sigma_\epsilon$ and $\kappa_\epsilon$ denote its scalar curvature.

We shall show that, for any continuous compactly supported function $\psi$ on $\bbE$:
\begin{equation}
\int_{\bbE} \psi \kappa_\epsilon \mu_\epsilon \to \int_{\bbF} d \psi.
\end{equation}
On the right hand side we integrate $\psi$ on $\bbF$ with respect to the metric induced by $\rho$ on $\bbF$, which is well defined. In other words we show that the densitized scalar curvature $\kappa_\epsilon \mu_\epsilon$ converges in the sense of measures to a certain measure supported on $\bbF$, given by the deficit angle.

To obtain this, one would like to apply the previous type of arguments to some two-dimensional space transverse to $\bbF$. However if one just chooses an arbitrary transverse plane the expression of scalar curvature induced in it will be difficult to relate to the scalar curvature on $\bbE$ and the deficit angle. Another idea would be to fix $x \in \bbF$ and take, inside sector $i+ 1/2$, the positive  cone, consisting of points $x + \bbR_+ m_i + \bbR_+ m_{i+1}$. But the union will be a piecewise linear cone, so that we would be on shaky grounds for doing calculus involving non-linear expressions with second order derivatives.

Our solution to this problem is an appeal to Frobenius' theorem concerning integrability of subbundles, as can be found for instance in \cite{Tay96I} page 40.

We fix an $\epsilon$ until further notice. For each point $x \in \bbE$ we let $\bbD_\epsilon(x)$ be the two-dimensional space orthogonal to $\bbF$ with respect to $\sigma_\epsilon[x]$:
\begin{equation}
\bbD_\epsilon(x) = \{ e \in \bbE \ : \ \forall f \in \bbF \quad \sigma_\epsilon[x](e,f) = 0 \}
\end{equation}
We shall prove that this subbundle is integrable, in the sense that through each point $x \in \bbE$ there passes a two-dimensional smooth manifold $S_\epsilon (x)$, having $\bbD_\epsilon(y)$ as tangent space at each $y \in S_\epsilon (x) $.

Some lemmas:
\begin{lemma}
Choose $f \in \bbF$. Then the (smooth) one-form $x \mapsto \sigma_\epsilon [x](f, \cdot)$ is closed.
\end{lemma}
\begin{proof}
If we look at the one-form $\rho(f, \cdot)$, we notice that it is constant in each sector and continuous at the hyperplane interfaces on tangential vectors. Therefore its exterior derivative in the sense of distributions is $0$.

Since $\sigma_\epsilon(f, \cdot) $ is obtained from $\rho(f, \cdot)$ by smoothing by convolution, it follows that it is also closed.
\end{proof}

Let $\nabla^\epsilon$ denote the Levi-Civita connection of the metric $\sigma_\epsilon$.

\begin{lemma} \label{lem:tor} For any $f\in \bbF$, considered as a translation invariant vectorfield on $\bbE$, $\nabla^\epsilon f = 0$.
\end{lemma}

\begin{proof}
Let $a, b$ denote two translation invariant vector fields on $\bbE$. We write:
\begin{align}
2 \sigma_\epsilon( \nabla^\epsilon_a f, b) & = \partial_a \sigma^\epsilon(f, b) - \partial_b \sigma^\epsilon(f, a),\\
& =  (\rmd \sigma^\epsilon(f, \cdot ))(a, b),\\
&  = 0.
\end{align}
We used first the Koszul formula for the Levi-Civita connection (\cite{Tay96I} equation (11.22) page 48). We noticed that commutators vanish and also that $\sigma^\epsilon(a, b)$ is invariant under translation by $f$, so that the term $\partial_f \sigma^\epsilon(a, b)$ vanishes. Secondly, we used an identity for the exterior derivative of one-forms (\cite{Tay96I} equation (13.55) page 69) which lets us apply the previous Lemma.

Since this holds pointwise, for all $a,b \in \bbE$, the lemma follows.

\end{proof}

Let $e' \in \bbE$ denote a vector not in $\bbF$ and $e''\in \bbE$ a vector not in $\bbF + \bbR e'$. We deduce a basis for $\bbD(x)$ as follows:
\begin{align}
e_1(x) & = e' - \calP_{\bbF}[x] e', \label{eq:eone} \\
e_2(x) & = e'' - \calP_{\bbF + \bbR e'}[x] e'',\label{eq:etwo}
\end{align}
where for instance $\calP_{\bbF}[x]$ denotes the orthogonal projection onto $\bbF$ with respect to $\sigma_\epsilon[x]$.

Then $(e_1(x), e_2(x))$ is a basis for $\bbD_\epsilon(x)$ and it is invariant with respect to translations along vectors in $\bbF$.
\begin{proposition} \label{prop:com} The commutator $[e_1, e_2]$ is in $\bbD$.
\end{proposition}

\begin{proof} Let $f$ be a vector in $\bbF$.
We write:
\begin{align}
\sigma_\epsilon( [e_1, e_2], f ) & = \sigma_\epsilon (\nabla^\epsilon_{e_1} e_2 - \nabla^\epsilon_{e_2} e_1, f ),\\
& = - \sigma_\epsilon( \nabla^\epsilon_{e_1} f, e_2 ) +  \sigma_\epsilon( \nabla^\epsilon_{e_2} f, e_1 ) ,\\
& = 0.
\end{align}
We used first torsion freeness, then metric compatibility of the Levi-Civita connection and finally the previous Lemma.
\end{proof}

Summing up, the situation is as follows. By Proposition \ref{prop:com} the subbundle $\bbD$ is integrable and yields a foliation of $\bbE$ by surfaces $S_\epsilon(x)$ ($x \in \bbE$), whose tangent planes are orthogonal to $\bbF$ at every point, with respect to $\sigma_\epsilon$. The foliation is invariant under translation along vectors in $\bbF$, and for each $y \in \bbE$ there is a unique  $x \in \bbF$ such that $y \in S_\epsilon(x)$. By Lemma \ref{lem:tor} these two-dimensional submanifolds have zero extrinsic curvature. The scalar curvature of the smoothed metric $\sigma^\epsilon$ on $\bbE$ is therefore related to the scalar curvature of the submanifolds $S_\epsilon$ equipped with the induced metric, in a very simple way. See for instance \cite{Wal84} page 464, for the case where $S_\epsilon$ is a hypersurface, that is, when $\bbE$ is three-dimensional.

We have an analogue of Lemma \ref{lem:nablam}:

\begin{lemma}\label{lem:mtransp}
In the union of the sectors $i - 1/2$ and $i+ 1/2$, consider the open subset $U_i^\epsilon$ of points whose distance to the boundary is strictly larger than $\epsilon$.

On $U_i^\epsilon$ we have $\nabla^\epsilon m_i = 0$.
\end{lemma}

\begin{proof}
 For this proof we fix $\epsilon$, and write  $\sigma = \sigma_\epsilon$ and $\nabla = \nabla^\epsilon$. We follow quite closely the proof of Lemma \ref{lem:tor}.

Let $a$ and $b$ be any two vectors in $\bbE$, considered as constant (translation invariant) vectorfields on $\bbE$.
We write:
\begin{equation}
2 \sigma (\nabla_a m_i, b ) = \partial_a \sigma (m_i, b) + \partial_{m_i} \sigma (a,b) - \partial_b \sigma(m_i, a).
\end{equation}

Consider the second term on the right hand side. We notice that  $\rho$ is invariant with respect to translations in direction $m_i$, as long as one stays inside the union of the two sectors $i-1/2$ and $i+1/2$. Therefore $\sigma$ also has this property as long as one stays inside $U_i^\epsilon$. Hence the second term vanishes.

Then, for the two remaining terms, we recognize:
\begin{equation}
\partial_a \sigma (m_i, b)  - \partial_b \sigma (m_i, a) = (\rmd \sigma(m_i, \cdot))(a,b).
\end{equation}
Since the one-form $\rho(m_i, \cdot)$ is constant in each of the two sectors  $i-1/2$ and $i + 1/2$, and continuous across the interface $i$, when evaluated on tangential vectors, we have that $\rmd \rho(m_i, \cdot) = 0$. Since $\sigma (m_i, \cdot)$ is deduced from $\rho(m_i, \cdot)$ by smoothing by convolution, it is also closed. 

We get:
\begin{equation}
\sigma (\nabla_a m_i, b ) = 0.
\end{equation}

Since this holds for all $a,b \in \bbE$, the proposition follows.
\end{proof}

We apply this to locate the curvature of the induced metrics on the manifolds $S_\epsilon$. For this discussion we fix one such surface.

\begin{proposition}\label{prop:curloc}
Inside the manifolds $S_\epsilon$, the curvature is located within distance $\calO(\epsilon)$ to the intersection of $S_\epsilon$ with the hinge $\bbF$.
\end{proposition}
\begin{proof}
We can define vectors $(e_1(x), e_2(x))$ as in (\ref{eq:eone}, \ref{eq:etwo}) starting with $e'= m_i$. We notice that then $e_1(x) = m_i$, for $x$ in the domain $U_i^\epsilon$ defined in the previous Lemma. Since there we have $\nabla^\epsilon m_i = 0$, there is no curvature in $S_\epsilon \cap U_i^\epsilon$.
\end{proof}
Next we evaluate the integral of densitized curvature inside $S_\epsilon$.
\begin{proposition}\label{prop:intkappabis}
For the metric induced by $\sigma_\epsilon$ in $S_\epsilon$, the densitized scalar curvature has integral equal to the deficit angle, defined in equation (\ref{eq:deficit}).
\end{proposition}
\begin{proof}
We may extend the techniques used in the two-dimensional setting using framefields adapted to $S^\epsilon$, as was already done in Proposition \ref{prop:curloc}.

Firstly, Proposition \ref{prop:holone} carries over by choosing an orthonormal framefield $(e_1, e_2)$ adapted to $S_\epsilon$, and integrating over a domain $T$ inside $S_\epsilon$.

Secondly, Proposition \ref{prop:holtwo}, also extends. We choose points $p_{i+1/2}$ on $ S_\epsilon$ inside sector $i+1/2$, at distance at least $\epsilon$ from the interfaces. We join them by curves in $S_\epsilon$ enclosing a two-dimensional domain $T$ inside $S_\epsilon$, containing the support of the curvature. The vectors $n_i^\pm$ appearing in the proof of Proposition \ref{prop:holtwo} are chosen such that $(m_i,n_i^\pm)$ is tangent to $S_\epsilon$ at $p_{i\pm 1/2}$, in addition to being orthonormal and oriented for $\rho_{i \pm 1/2}$ as before. Notice that, at the point $p_{i\pm 1/2}$, $\sigma_\epsilon$ equals $\rho_{i \pm 1/2}$. Notice also that, by Lemma \ref{lem:mtransp}, the vector $m_i$ is parallel-transported to itself, along the chosen curve from $p_{i-1/2}$ to $p_{i+1/2}$. Therefore all the previous arguments carry over to the more general setting.

We then conclude exactly as in Proposition \ref{prop:intkappa}.
\end{proof}

Finally we want to deduce results on the whole space $\bbE$. First some estimates: 
\begin{proposition}\label{prop:estpoint}
The curvature of $\sigma_\epsilon$ is supported in a tube around $\bbF$ of radius $\calO(\epsilon)$. Moreover, the  curvature is bounded by $\calO(\epsilon^{-2})$ pointwise, uniformly in space.
\end{proposition}

\begin{proof}
For the first assertion, we use Lemmas \ref{lem:tor} and \ref{lem:mtransp}. They show that there is no curvature on the domains of type $U_i^\epsilon$. The union of these domains contain the exterior of tubes around $\bbF$ with radius $C\epsilon$, for $C$ sufficiently large.

The second assertion follows from explicit coordinate expressions for the curvature (e.g. \cite{Tay96I} equation (3.6) page 469), taking into account that, since the Regge metric $\rho$ is bounded, the partial derivatives of order $k$ of the smoothed metric $\sigma_\epsilon$ are $\calO(\epsilon^{-k})$ pointwise, uniformly in space. The inverse of $\sigma_\epsilon$ also remains bounded, by positive definiteness of $\rho$.
\end{proof}

\begin{proposition} Let $d$ be the deficit angle. We have, as $\epsilon \to 0$:
\begin{equation}
\int_{\bbE} \psi \kappa_\epsilon \mu_\epsilon \to \int_{\bbF} d \psi.
\end{equation}
\end{proposition}
\begin{proof}
Let $\omega$ be a modulus of continuity for $\psi$.

Let $\kappa_\epsilon'$ and $\mu_\epsilon'$ denote the scalar curvature and volume form on the submanifolds $S_\epsilon$ of the metric induced by $\sigma_\epsilon$. Let  $\mu_\bbF$ be the volume form on $\bbF$ induced by $\rho$. Notice that the smoothed metrics $\sigma_\epsilon$ agree everywhere with $\rho$ when evaluated on two vectors in $\bbF$. Because the extrinsic curvature is $0$, we have:
\begin{align}
\int_{\bbE} \psi \kappa_\epsilon \mu_\epsilon & = \int_{\bbF} \mu_\bbF (x) \int_{S_\epsilon(x)} \psi \kappa_\epsilon' \mu_\epsilon',\\
& = \int_{\bbF} \mu_\bbF (x) \int_{S_\epsilon(x)} \psi(x)\kappa_\epsilon' \mu_\epsilon' +   \int_\bbF \mu_\bbF(x) \int_{S_\epsilon(x)} (\psi - \psi(x)) \kappa_\epsilon' \mu_\epsilon', \label{eq:twoterms}\\
& = \int_{\bbF} \mu_\bbF (x) d \psi(x) + \calO(\omega(\epsilon)).
\end{align}
In equation (\ref{eq:twoterms}), we used Proposition \ref{prop:intkappabis} to evaluate the first term and Proposition \ref{prop:estpoint} to estimate the second.

This concludes the proof.
\end{proof}

Finally we conclude:
\begin{theorem}
For a Regge metric on a simplicial complex placed in a Euclidean space, if we smooth it by convolution, with parameter $\epsilon$, the densitized scalar curvatures converge in the sense of measures, to the measure defined by Regge calculus (supported on hinges and defined by the deficit angles), as $\epsilon \to 0$.
\end{theorem}
\begin{proof}
As in Proposition \ref{prop:estpoint}, the curvature is located in a tubular neighborhood of the codimension two skeleton with radius $\calO(\epsilon)$, and is pointwise bounded by $\calO(\epsilon^{-2})$. We proceed as in the previous proposition, with an additional contribution from the codimension three skeleton, which is $\calO(\epsilon)$.

\end{proof}

\section*{Acknowledgements}
Stimulating discussions with Tore G. Halvorsen, concerning Lattice Gauge Theory, are gratefully acknowledged.

This work was supported by the European Research Council through the FP7-IDEAS-ERC Starting Grant scheme, project 278011 STUCCOFIELDS.

\bibliography{../Bibliography/alexandria,../Bibliography/mybibliography}{}
\bibliographystyle{plain}

\end{document}